\documentclass[11pt]{article}
\usepackage[utf8]{inputenc}
\usepackage{multirow}

\usepackage{cite}
\usepackage{lineno}
\usepackage{ifthen}
\usepackage{makeidx}
\usepackage{amsmath,amssymb,amstext} 
\usepackage[pdftex]{graphicx} 
\usepackage{fullpage}
\usepackage[ddmmyyyy,hhmmss]{datetime}
\numberwithin{equation}{section}
\numberwithin{figure}{section}
\numberwithin{table}{section}

\usepackage[pdftex,letterpaper=true,pagebackref=true]{hyperref} 
\hypersetup{
    plainpages=false,       
    pdfpagelabels=false,     
    bookmarks=true,         
    unicode=false,          
    pdftoolbar=true,        
    pdfmenubar=true,        
    pdffitwindow=false,     
    pdfstartview={FitH},    
    pdftitle={
FR and SD for Cone Optimization
},    
    pdfnewwindow=true,      
    colorlinks=true,        
    linkcolor=blue,         
    citecolor=green,        
    filecolor=magenta,      
    urlcolor=cyan           
}

\usepackage{comment}

\usepackage{float}
\usepackage{bbm}
\usepackage{exscale}
\usepackage{tabularx}
\usepackage{syntonly}
\usepackage{lineno}

\usepackage{algorithm,algorithmic}
\usepackage{amssymb}
\usepackage{amsmath}
\usepackage{amsthm}
\usepackage{latexsym}
\usepackage{makeidx}
\usepackage{longtable}
\usepackage[noabbrev]{cleveref}

\makeindex

\usepackage{chngcntr}

\def\ri{\text{ri\,}}

\def\R{\mathbb{R}}

\def\Sc{\mathbb{S}}

\def\Snp{\Sc_+^n}

\def\Snpp{\Sc_{++}^n}

\def\Rnp{\mathbb{R}_+^n}

\def\cK{\mathcal{K}}

\def\eqref#1{{\normalfont(\ref{#1})}}

\def\EDMp{\mbox{\boldmath$EDM$}}

\def\eqref#1{{\normalfont(\ref{#1})}}

\newtheorem{theorem}{Theorem}[section]

\newtheorem{definition}[theorem]{Definition}

\newtheorem{proposition}[theorem]{Proposition}

\newtheorem{corollary}[theorem]{Corollary}

\newtheorem{remark}[theorem]{Remark}

\newtheorem{lemma}[theorem]{Lemma}

\crefname{thm}{Theorem}{Theorems}
\Crefname{thm}{Theorem}{Theorems}
\crefname{assump}{Assumption}{Theorems}
\Crefname{assump}{Assumption}{Theorems}
\crefname{problem}{Problem}{Theorems}
\Crefname{problem}{Problem}{Theorems}
\crefname{conjecture}{Conjecture}{Theorems}
\Crefname{conjecture}{Conjecture}{Theorems}
\crefname{proposition}{Proposition}{Propositions}
\Crefname{proposition}{Proposition}{Propositions}
\crefname{prop}{Proposition}{Propositions}
\Crefname{prop}{Proposition}{Propositions}
\crefname{cor}{Corollary}{Corollaries}
\Crefname{cor}{Corollary}{Corollaries}
\crefname{lem}{Lemma}{Lemmas}
\Crefname{lem}{Lemma}{Lemmas}
\theoremstyle{definition}
\crefname{definition}{definition}{definitions}
\Crefname{definition}{Definition}{Definitions}
\crefname{defn}{definition}{definitions}
\Crefname{defn}{Definition}{Definitions}
\crefname{remark}{Remark}{Remarks}
\Crefname{remark}{Remark}{Remarks}
\crefname{rmk}{Remark}{Remarks}
\Crefname{rmk}{Remark}{Remarks}
\crefname{example}{Example}{Examples}
\Crefname{example}{Example}{Examples}
\crefname{align}{}{}
\Crefname{align}{}{}
\crefname{equation}{}{}
\Crefname{equation}{}{}

\makeatletter
\newcommand*\bigcdot{\mathpalette\bigcdot@{.5}}
\newcommand*\bigcdot@[2]{\mathbin{\vcenter{\hbox{\scalebox{#2}{$\m@th#1\bullet$}}}}}
\makeatother

\newcommand{\textdef}[1]{\textit{#1}\index{#1}}
\newcommand{\Ss}{\mathcal{S}}

\newcommand{\<}{\langle}
\renewcommand{\>}{{\rangle}}

\newcommand{\bE}{{\mathbb E} }
\newcommand{\bF}{{\mathbb F} }

\newcommand{\cL}{{\mathcal L} }
\newcommand{\cA}{{\mathcal A} }

\newcommand{\cM}{{\mathcal M} }

\newcommand{\cF}{{\mathcal F} }

\newcommand{\E}{{\mathbb E} }

\newcommand{\cN}{{\mathcal N} }

\newcommand{\cE}{{\mathcal E} }
\newcommand{\cEn}{\cE^n}

\newcommand{\DNN}{\textbf{DNN}\,}
\newcommand{\DNNp}{\textbf{DNN}}

\newcommand{\SDPp}{\textbf{SDP}}

\newcommand{\bbm}{\begin{bmatrix}}
\newcommand{\ebm}{\end{bmatrix}}
\newcommand{\bem}{\begin{pmatrix}}
\newcommand{\eem}{\end{pmatrix}}
\newcommand{\beq}{\begin{equation}}
\newcommand{\beqs}{\begin{equation*}}
\newcommand{\bet}{\begin{table}}
\newcommand{\eeq}{\end{equation}}
\newcommand{\eeqs}{\end{equation*}}
\newcommand{\beqr}{\begin{eqnarray}}

\DeclareMathOperator{\face}{face}
\DeclareMathOperator{\sd}{sd}

\DeclareMathOperator{\Null}{null}

\DeclareMathOperator{\range}{range}

\DeclareMathOperator{\trace}{{trace}}

\DeclareMathOperator{\diag}{{diag}}
\DeclareMathOperator{\Diag}{{Diag}}

\DeclareMathOperator{\offDiag}{{offDiag}}

\DeclareMathOperator{\relint}{{relint}}

\DeclareMathOperator{\rank}{{rank}}

\DeclareMathOperator{\conv}{{conv}}

\newcommand{\nc}{\newcommand}
\nc{\arrow}{{\rm arrow\,}}
\nc{\Arrow}{{\rm Arrow\,}}
\nc{\BoDiag}{{\rm B^0Diag\,}}
\nc{\bodiag}{{\rm b^0diag\,}}

\nc{\Mm}{{\mathcal M}^{m} }
\nc{\Mmn}{{\mathcal M}^{mn} }
\nc{\Mnr}{{\mathcal M}_{nr} }
\nc{\Mnmr}{{\mathcal M}_{(n-1)r} }
\nc{\kwqqp}{Q{$^2$}P\,}
\nc{\kwqqps}{Q{$^2$}Ps}

\nc{\notinaho}{(X,S)\in \overline{AHO}(\A)}
\nc{\inaho}{(X,S)\in AHO(\A)}

\newcommand{\bea}{\begin{eqnarray}}%
\newcommand{\eea}{\end{eqnarray}}%
\newcommand{\beas}{\begin{eqnarray*}}%
\newcommand{\eeas}{\end{eqnarray*}}%
\newcommand{\Int}{{\rm int\,}}
\newcommand{\cone}{{\rm cone\, }}

{}

\newcommand{\Hnp}[1][]{\,\mathbb{H}_+^{\ifthenelse{\equal{#1}{}}{n}{#1}}}
\newcommand{\Hkp}[1][]{\,\mathbb{H}_+^{\ifthenelse{\equal{#1}{}}{k}{#1}}}
\newcommand{\Hnpp}[1][]{\,\mathbb{H}_{++}^{\ifthenelse{\equal{#1}{}}{n}{#1}}}
\newcommand{\Hn}[1][]{\,\mathbb{H}^{\ifthenelse{\equal{#1}{}}{n}{#1}}}
\newcommand{\Hk}[1][]{\,\mathbb{H}^{\ifthenelse{\equal{#1}{}}{k}{#1}}}
\newcommand{\Dn}[1][]{\,\mathbb{D}^{\ifthenelse{\equal{#1}{}}{n}{#1}}}

\title{Singularity degree of non-facially exposed faces}
\author{Fei Wang\thanks{Department of Combinatorics and Optimization, Faculty of Mathematics, University of Waterloo, Waterloo, Ontario, Canada N2L 3G1}
 \and \href{http://www.math.uwaterloo.ca/~hwolkowi/} {Henry Wolkowicz}%
    \thanks{Department of Combinatorics and Optimization, Faculty of Mathematics, University of Waterloo, Waterloo, Ontario, Canada N2L 3G1; Research supported by The Natural Sciences and Engineering Research Council of Canada.}
}

\begin{document}

\maketitle

\begin{abstract}
In this paper, we study the facial structure of the linear image of a convex cone. 
We define the singularity degree of a face of a convex cone to be the minimum 
number of steps it takes to expose this face using exposing vectors from the 
dual cone. We show that the singularity degree of the linear image of a 
convex cone is exactly the number of facial reduction steps it takes to obtain the 
minimal face containing the feasible set in the corresponding linear conic optimization problem. Our 
result generalizes the relationship between the complexity of general 
facial reduction algorithms and facial exposedness of conic images 
under a linear transform by Drusvyatskiy, Pataki and Wolkowicz to arbitrary singularity degree. We present our results in the original form and also in its nullspace form. As a major application, we derive an upper bound for the singularity degree of generic frameworks or tensegrities in 2 or 3 dimensional space underlying certain graphical structure. 
\end{abstract}



\tableofcontents
\listofalgorithms   
\addtocontents{loa}{\def\string\figurename{Algorithm}}
\listoftables
\listoffigures

\section{Introduction}
We consider the following linear \textdef{conic optimization problem} 
in its primal form,
\begin{equation}\label{prob:main}
  (P) \quad \begin{array}{cc}
        \max  & \langle C, X \rangle  \\
          \mbox{s.t. } &  \mathcal{M}(X) = b\in \bF \\
          & X \in \mathcal{K}\subseteq \bE,
    \end{array}
\end{equation}
where 
$\mathcal{M}:\bE\to \bF$ is a surjective linear transformation between
finite dimensional Euclidean spaces, $\<\cdot,\cdot\>$ denotes the
respective inner product, and $\mathcal{K}\subset \bE$ is a convex cone.


\index{$\cF$, feasible region}
The \textdef{feasible region, $\cF$}, of this conic 
linear optimization problem is denoted
\begin{equation}
\label{eq:feasset}
\cF = \{ X \in \cK : \cM (X) = b\}.
\end{equation}
Understanding the structure of the image of the convex cone $\cK$ under a 
linear transformation is an important problem in optimization. 
Since $\cM$ is surjective and hence an \emph{open mapping}, strict
feasibility (also called Slater's constraint qualification, 
$\exists \hat X\in \Int K, \cM(\hat X)=b$) requires that $b\in \Int
\cM(K)$.
In addition, the linear image of a convex cone can have very different
structure from the convex cone itself. The linear image of a closed
convex cone is not necessarily closed, see~\cite{Pataki:07} for
characterizations. The convex cone $\cK$ can be
\textdef{facially exposed}, i.e., any 
face $f$ of $\cK$ can be written as $f =
\cK \cap v^{\perp}\unlhd \cK$ with $v$ in the 
\textdef{dual cone, $\cK^*$}, of $\cK$, while 
simultaneously the linear image of $\cK$ is \emph{not} facially exposed.

\index{$\cK^*$, dual cone}





The smallest number of steps of facial reduction required to obtain the
minimum face containing the feasible region is called the \emph{singularity degree} of the system
\eqref{eq:feasset} over $\cK$. 
The singularity degree is an intriguing measurement of complexity for
$\cF$, see
e.g., the case when $\cK$ is the cone of positive semidefinite matrices~\cite{sturm2000error,sremac2021error}.  
Motivated by these results, in this paper we aim to generalize the
notion of singularity degree to the
case of the linear image of $\cK$ and in particular, for the
case of not facially-exposed cones. We include results on extending
the coordinate shadow theorem \cite{DrPaWo:14} of singularity degree one
to higher singularity degree.  Therefore
we obtain an alternative characterization of the singularity degree of a
linear optimization problem. We are also motivated by
the work in ~\cite{connelly2015iterative}  on
the universal rigidity of frameworks where the connections between
facial reduction and stress matrices are discussed.

For a set $S$, we let $\face(S,\cK)$ denote the smallest face of $\cK$
containing the set $S$.
We note the following related result, reworded in our notation.
\begin{theorem}{\cite[Theorem 4.1]{DrPaWo:14}} Let $\cK$ be closed. Then
for any vector $v$ exposing $\face(b,\cM(\cK))$, the vector $\cM^*v$
exposes $\face(\cF,\cK)$.
\end{theorem}
As mentioned in \cite{DrPaWo:14}, ``there is difficulty when working with
$\cM(K)$ as it is not
usually a well-understood set; and systematically recognizing points on 
its boundary is hopeless and exposing vectors are out of reach.''
This is the motivation for our study  in this paper.


\section{Background: facial structure of a convex cone}
We begin with notation and then present some preliminary results.

\subsection{Notation}\label{subsec:notation}
Consider two Euclidean spaces $\bE$ and $\bF$ with inner products
$\langle \cdot , \cdot  \rangle_{\bE}$ and $\langle \cdot , \cdot
\rangle_{\bF}$, respectively. We use $\langle \cdot, \cdot\rangle$ when
the meaning is clear.
Define $\cM: \bE \rightarrow \bF$ to be a linear surjective mapping
between the Euclidean spaces $\bE$ and $\bF$. The \textdef{adjoint of
$\cM$, $\cM^*$}, is defines as the unique linear mapping from $\bF$ to $\bE$
such that $\langle \cM(x), y \rangle_{\bF} = \langle x, \cM^*(y)
\rangle_{\bE}, \forall x \in \bE, \forall y\in \bF$. Consider a convex
cone $\cK \subset \bE$. We denote $\Int(\cK),\ri(\cK)$ to be the 
interior and relative interior
of $\cK$, respectively; and we denote the orthogonal complement of $\cK$ as
$\cK^{\perp}$. For a vector $v$, we let $v^{\perp}:= \{ v \}^{\perp}$.
We associate to $\cK$ its \textdef{dual cone}
\[
\cK^* = \{ y \in \bE: \langle y, x \rangle \geq 0,\,  \forall  x \in \cK \}.
\]

\index{face, $f\unlhd K$}
\index{$f\unlhd K$, face}
A convex subset $f \subseteq \cK$ is called a \emph{face} of $\cK$, 
denoted $f \unlhd \cK$, if 
\[
x, y \in \cK, \,
\frac{1}{2} (x + y) \in f \implies x \in f, y \in f.
\]
In other words, a face $f$ contains all line segments in $\cK$ whose relative
interior intersects $f$. The \emph{minimal face} containing a set
$S \in \cK$, denoted \emph{$\face(S, C)$}, is the intersection of all
faces of $\cK$ containing $S$. A face $f$ of $\cK$ is an
\textdef{exposed  face} when there exists a vector $v \in \cK^*$ such
that $\cF = \cK \cap v^{\perp}$. In this case, we say $v$ exposes $\cF$.
The cone $\cK$ is called \textdef{facially exposed} if all faces of
$\cK$ are exposed. Examples of facial exposed cones includes: the
nonnegative orthant $\Rnp$; 
the cone of positive semidefinite matrices $\Snp$; and the
cone of Euclidean distance matrices $\cEn$. The linear image of the
(non-polyhedral) cones $\Snp,\cEn$ are not necessarily facially exposed,
see~\cite{DrPaWo:14}.
\index{minimal face, $\face(S,C)$}
\index{$\face(S,C)$,  minimal face}


\subsubsection{\SDPp,  \EDMp \; cones}
We denote $\Ss^n$ as the vector space of $n \times n$ real symmetric
matrices. The inner product of two symmetric matrices is defined to be
the trace $\langle A, B \rangle = \trace(AB)$.  The cones of positive
semidefinite and positive definite  matrices in $\Ss^n$ 
are denoted as $\Snp$ and $\Snpp$, respectively. They give rise to the
L\"owner partial order
\[
X\succeq Y \iff Y-X\in \Snp, \quad
X\succ Y \iff Y-X\in \Snpp.
\]
For a general cone, we use $X\succeq_K Y, X\succ_K Y$.
The cone of positive semidefinite matrices is facially exposed. Let 
$X\in \relint f, \, f\unlhd \Snp$,
and $X=UDU^T$ be its compact (orthogonal) spectral decomposition. Then,
see e.g.,~\cite{DrusWolk:16},
\[
f 
= \left\{ U 
\begin{bmatrix}
    R & 0 \\
    0  & 0 
\end{bmatrix}
 U^T: R \in \Ss_+^r\right\}
= \Snp \cap \left( V V^T \right)^{\perp},
\]
where $U$ is an orthogonal matrix and $r = \rank(X)$. Here $VV^T$ is the
exposing vector for the face, i.e., the columns of $V$ form a basis for
the orthogonal complement of the range of $U$.

A matrix $D \in \Ss^n$ is called an \emph{Euclidean distance matrix}
(EDM for short) if there exists $n$ points $p_i, i=1,...,n$ in
$\mathbb{R}^k$ such that $D_{ij} = ||p_i - p_j||^2$ for all indices $i,j$. We denote $\cEn$ as the cone of all $n\times n$ Euclidean distance matrices. The EDM cone $\mathcal{E}^n$ is linearly isomorphic to $\Ss^{n-1}_+$. 
See e.g.,~the book~\cite{MR3887551}.
The EDM matrices and PSD matrices
(Gram matrices) are connected by the \textdef{Lindenstrauss mapping, $K$}. There is an one-one correspondence between the centered PSD matrices and EDM matrices by Lindenstrauss mapping. 
We state the result as follows,
\index{$K$, Lindenstrauss mapping}
\[
K:=\Ss^n \rightarrow \Ss^n,
K(X)_{ij}:= X_{ii} + X_{jj} - 2X_{ij}.
\]
with adjoint and \emph{Moore-Penrose pseudoinverse},
\[
K^*(D) = 2(\Diag(De) - D),\quad K^{\dagger}(D) = -\frac{1}{2} 
J_n \offDiag(D) J_n,
\]
respectively. Here $\Diag$ is the adjoint of $\diag$, the matrix $J_n := I - \frac{1}{n} ee^T$ is the orthogonal
projection onto SC, and $\offDiag(D)$ refers to zeroing out the diagonal of $D$.

\subsection{Facial reduction}
In ~\cite{DrPaWo:14}, the facial exposedness of the linear image of
$\cK$ is shown to be equivalent to the singularity
degree,~\cite{sturm2000error,sremac2021error}, of~\eqref{prob:main} being one.
Note that the singularity degree is the minimum number of steps needed
when applying a general facial reduction 
algorithm,~\cite{borwein1981regularizing}, on a conic optimization problem.

A facial reduction algorithm aims to resolve lack of \textdef{strict
feasibility}
in the system \eqref{prob:main}.  The lack of strict feasiblity can
cause many problems in an optimization problem. For example, one can
have a positve duality gap, strong
duality can fail  to hold, and stability and numerical issues can arise.
The facial reduction algorithm is motivated by the
following \textdef{theorem of alternative} in conic optimization, 
see e.g.,~\cite{borwein1981regularizing}:
\begin{lemma}
\label{lem:thmalt}
Exactly one of the following two alternatives hold:
\begin{enumerate}
\index{strict feasibility}
    \item Strict feasibility holds in \cref{eq:feasset} 
$(\exists \hat X \in \ri (\cK) \mbox{ such that } \cM(\hat X) = b)$.
    \item 
There exists $v\in \bF$ such that:
\begin{equation}
\label{eq:auxsyst}
\cK^{\perp} \not\ni Z = \cM^*v \in \cK^*,  \mbox{ and } \langle v,b \rangle = 0.
\end{equation}
\end{enumerate}
\end{lemma}
We call $Z\in \cK^*$ an \textdef{exposing vector} for the face $f
\unlhd K$ if 
$X\in f \implies \langle Z,X\rangle = 0$.
We see that the so-called \textdef{auxiliary system}~\cref{eq:auxsyst}
is equivalent to the existence of the special exposing vector $Z$.
 One can then replace $\cK$ by the \textdef{exposed face} $\cK \cap
Z^{\perp}\unlhd \cK$ and obtain a smaller problem with a smaller cone. 
This also results in a reduction of the number of constraints
in~\cref{eq:feasset} to maintain surjectivity,
 see~\cite{borwein1981regularizing}.
This procedure can be repeated until strict feasibility is attained.
We summarize the facial reduction algorithm over a convex cone in
\Cref{algo:frnonexp1}, see~\cite{borwein1981regularizing,DrusWolk:16}:
\begin{algorithm}
\caption{facial reduction}
\begin{algorithmic}\label{algo:frnonexp1}
\STATE \textbf{Input:} $(P)$\\
\STATE \textbf{Output:} A sequence of reducing certificates $Z_i, \,
i=0,1,\ldots,k$\\
\STATE  Set $i \leftarrow 0$, let $\cF_i \leftarrow \cK, \,
i=0,1,\ldots,k$, 
\STATE \textbf{Step 1:} Find an exposing vector $  Z_i = \cM^* v_i \in (\cF_i)^*$, $Z_i \notin \cF_i^{\perp},$ $\langle v_i, b \rangle = 0$. 
\STATE If no such $Z_i$ exists, then exit.
\STATE Set $\cF_{i+1} \leftarrow \cF_i \cap Z_i^{\perp}$
\STATE $i = i + 1$
\STATE Return to Step 1.
\STATE Output: the minimal face $\cF_k$ of $\cK$ containing the feasible region of $(P)$. 
\end{algorithmic}
\end{algorithm}




Recall that the smallest number of steps of 
facial reduction required to obtain the
mimimum face is called the \emph{singularity degree} of the system
\eqref{eq:feasset} over $\cK$.






The SDP cones and EDM cones are ``nice" cones\cite{Pataki:07}. A closed convex cone $\cK$ is nice if $\cK^* + E^{\perp}$ is closed for any face $E$ of $\cK$. Due to this "niceness", one can construct a family of extended duals such that strong duality holds by the facial reduction procedure\cite{pataki2013strong,ramana1997strong,ramana1997exact}.   The length of the sequence in the extended duals is exactly the singularity degree.

Two closely related problems in the literature are the EDM completion problem and PSD completion problem. In those problems, we are given a partial EDM or PSD matrix whose entries are specified only on a subset of positions  and the goal is to fill in the missing entries so that the resulting matrix is EDM or PSD. If we consider the position of a specified entry as an edge of a graph, then we can associate a graph $G$ to the projection.   It appears that the singularity degree of EDM and PSD problems are closely related to the graph $G$.  

The relation between the singularity degree of the completion problem and the underlying graph $G$ has been an active research area. In~\cite{DrPaWo:14}, it is shown that the singularity degree of EDM and PSD completion problem is at most 1 when the underlying graph $G$ is chordal. The EDM completion problem is equivalent to finding frameworks with the specified entries (distances) in Euclidean spaces while the PSD completion is equivalent to finding spherical frameworks on a unit sphere. In ~\cite{gortler2014generic}, it is shown that a generic universal rigid framework in an Euclidean space has singularity degree at most 1. In~\cite{tanigawa2017singularity}, the worse case singularity degree of a spherical framework underlying a graph $G$ is studied, and it is shown to be at most 1 if the $G$ is chordal. Also in~\cite{tanigawa2017singularity}, spherical frameworks whose singularity degree grows linearly in the number of vertices are constructed.

\section{Singularity degree of faces for non-facially exposed cones}
\label{sect:singdeg}
To formally define singularity degree of a face, we first start with an example. 
We consider a cone whose cross section is a half circle union a square
as shown in set $C, 0 \notin \conv(C)$, \Cref{fig:pic1}.
\begin{figure}[h]
    \centering
    \includegraphics[width=0.5\textwidth,height=0.2\textheight]{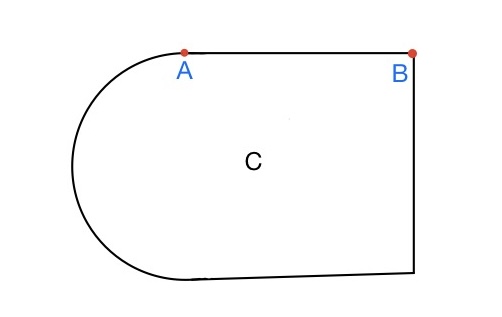}
    \caption{cross section of a non facially exposed cone}
    \label{fig:pic1}
\end{figure}
The cone $\cK= \cone(C)\subset \R^3$, i.e.,~the closed convex cone
generated by the compact \textdef{base}, $C$.

The ray (line) through the point $A$, $f_A:=\cone A\unlhd \cK$, 
represents a face of the convex cone $\cK=\cone C$.
Any exposed face of $\cK$ containing $A$ 
must also contain the line segment $AB$.
We can find an exposing vector $d_1 \in \cK^*$ such that $\cone(AB) =
\cone(C) \cap d_1^{\perp}$, and another exposing vector $d_2 \in (C \cap
d_1^{\perp})^*$ such that we get
\[
	\cone(A) = \cK \cap d_1^{\perp} \cap d_2^{\perp}.
\]
Hence we informally define the "singularity degree" of $A$ to be 2,
see~\Cref{def:singdegface} and \Cref{algo:frnonexp}, below.
\begin{algorithm}
\caption{singularity degree of  face $\cF_t$ of cone $\cK$}
\begin{algorithmic}\label{algo:frnonexp}
\STATE \textbf{Input:} A cone $\cK\subseteq \bF$, a finite dimensional Euclidean
space,\; a proper face $\cF_t \unlhd \cK$. \\
\STATE \textbf{Output:} Singularity degree $d$\\
\STATE  Set $i \leftarrow 0$, let $\cF_i \leftarrow \cK$, 
\STATE \textbf{Step 1:} Find an exposing vector $ d_i \in \cF_i^*$, $d_i \notin \cF_i^{\perp} $ and $d_i \perp \cF_t$
\STATE If no such $d_i$ exists, then exit and return $d =i$.
\STATE Set $\cF_{i+1} \leftarrow \cF_i \cap d_i^{\perp}$
\STATE $i = i + 1$
\STATE Return to Step 1.
\STATE Output $i$.
\end{algorithmic}
\end{algorithm}


The following lemma plays a key role in defining the singularity degree of a face.
\begin{lemma}[{\cite[Proposition 3.6]{borwein1981regularizing}}]
\label{lem:exp}
Let $\cK$ be a convex cone and $\cF \unlhd \cK$, be a proper face. 
Then $\cF \unlhd \cL \unlhd \cK$ for some \emph{exposed} face 
$\cL \unlhd \cK, \cL \neq \cK$.
\end{lemma}
\begin{proof}
We include a proof for completeness.
Assume $\cF$ is proper. Let $x \in \relint \cF$. Then, by definition of
a face, we get $x \in \partial \cK$, the boundary of $\cK$. 
\index{$\partial \cK$, boundary of $\cK$}
\index{boundary of $\cK$, $\partial \cK$}

By the hyperplane support version of the
Hahn-Banach separation theorem there is a hyperplane
supporting $\cK$ at $x$,~\cite{hiriart2004}. Therefore there exists $s \in \cK^*$ such that 
\[
\langle s, x \rangle = t\text{  and  }
\langle s, y \rangle \geq t,\, \forall  y \in \cK.
\]

We claim that $t = 0$.  If $t > 0$, then $\lim_{n \infty} \langle s,
\frac{x}{n} \rangle = 0 <t$, a contradiction. Also if $ t < 0$, then for
$n$ large enough, we have $\langle s, n x \rangle < t$, again a contradiction.

We now claim that $\langle s, \bar{x} \rangle = 0$ for any $\bar{x} \in
\cF$. Let $\bar{x} \in \cF$. Since $x \in \relint \cF$, we have $x =
\lambda_1 \bar{x} + (1-\lambda_1) x_1$ for some $x_1 \in \cF$ and
$\lambda_1 \in (0,1)$. Therefore $\langle s, x \rangle = 0$ implies
$\langle s, \bar{x} \rangle = \langle s, x_1 \rangle = 0$.

From this it follows that $\emptyset\neq \cF_1=s^{\perp}\cap \cK\subset\partial \cK$ and $\cF_1$ is an exposed face of $\cK$ containing $\cF$.


\end{proof}

\begin{theorem}
Algorithm~\ref{algo:frnonexp} terminates in a finite number of 
steps with $\cF_d = \cF_t$
\end{theorem}
\begin{proof}
Since $\cK$ is a cone in a finite dimensional Euclidean space, $\cK$ is finite dimensional. Since $d_i \notin \cF_i^{\perp}$ at each iteration of Algorithm~\ref{algo:frnonexp}, we have $\cF_{i+1}$ is a proper face of $\cF_i$, therefore $\cF_{i+1}$ is contained in a proper exposed face of $\cF_i$ by Lemma~\ref{lem:exp}, hence the dimension is dropped at least by 1. It is easy to see that $\cF_t \subseteq \cF_i$ at each step. If $\cF_t$ is a proper face of $\cF_d$, then there always exists an exposing vector $d_i \in \cF_d^*$, $d_i \notin \cF_d^{\perp} $ and $d_i \perp \cF_t$, a contradiction. Therefore $\cF_d = \cF_t$.
\end{proof} 

\begin{definition}
\label{def:singdegface}
The minimal number $d$ of  steps taken in Algorithm~\ref{algo:frnonexp} to  obtain $\cF_t$ is defined as the singularity degree of $\cF_t$ over $\cK$.
\end{definition}

The minimum number of steps taken in Algorithm~\ref{algo:frnonexp} can be achieved by a "greedy" strategy, i.e., we seek $d_i \in \ri (\cF_i^* \cap \{ \cF_t\})^{\perp}$. The following lemma shows the correctness of this strategy. We note a similar result about the facial reduction result is in~\cite[Proposition 13]{lourencco2018facial}.

\begin{lemma}\label{lem:min}
Assume $\bar{\cF} \subseteq \cF$, and let $\bar{d} \in \ri (\bar{F}^* \cap \{ \cF_t\}^{\perp})$ and $d \in F^* \cap \{ \cF_t\}^{\perp}$. Then $\bar{\cF}\cap \bar{d}^{\perp} \subseteq \cF\cap d^{\perp}$.
\end{lemma}
\begin{proof}
Since $\bar{\cF} \subseteq \cF$, we have $\cF^*  \subseteq \bar{\cF}^*$. Hence $d \in \bar{\cF}^*$ and $d \in \bar{F}^* \cap \{ \cF_t\}^{\perp}$. Since $\bar{d} \in \ri (\bar{F}^* \cap \{ \cF_t\}^{\perp}) $, there exists some $\alpha > 1$ such that 
\[
\alpha \bar{d} + (1 - \alpha) d \in \bar{F}^* \cap \{ \cF_t\}^{\perp}
\]

Let $x \in \bar{\cF} \cap \bar{d}^{\perp}$, then $\langle x, \alpha \bar{d} + (1 - \alpha) d \rangle \geq 0$. Since $\langle x, \bar{d} \rangle = 0$, we have $\langle x, (1-\alpha)d \rangle \geq 0$, so $\langle x, d \rangle \leq 0$. Since $d \in \bar{\cF}^*$ and $x \in \bar{\cF}$, we have $\langle x, d \rangle \geq 0$. Therefore $\langle x, d \rangle =0$. So $x \in d^{\perp}$ and we have $\bar{\cF}\cap \bar{d}^{\perp} \subseteq \cF\cap d^{\perp}$.
\end{proof}
We now have the following proposition following Lemma~\ref{lem:min}.
\begin{proposition}\label{prop:min}
The minimum number of steps taken in Algorithm~\ref{algo:frnonexp} is obtained by computing $d_i \in \ri (\cF_i^* \cap  \{\cF_t\} ^{\perp})$ in each iteration.
\end{proposition}
\begin{proof}
Direct consequence from Lemma~\ref{lem:min}.
\end{proof}
\subsection{Connection to linear image of a convex cone}

We now connect the singularity degree of a conic optimization problem with the singularity degree of the minimal face in the projected cone.
We follow the definition of  $\bE,\bF,\cM$  in section~\ref{subsec:notation}   
We start with the following Lemma
\begin{lemma}  
\label{lem:expose} 
With $\cK, \cM$ as above, we have:
\begin{enumerate}
    \item $\cM(\cK \cap (\cM^*y)^{\perp}) = \cM(\cK) \cap \{ y\}^{\perp}$ 
    \item $y_2 \in (\cM(\cK) \cap y_1^{\perp})^* \iff \cM^*y_2 \in (\cK \cap (\cM^*y_1)^{\perp})^*$ 
    \item 
\label{item:properexposedface}
$\cM^*y$ exposes a proper face of $\cK$ if and only if $y$ exposes a proper face of $\cM(\cK)$.
\end{enumerate}
\end{lemma}
\begin{proof}
\begin{enumerate}

    \item Let $z \in \cM(\cF \cap (\cM^*y)^{\perp})$, then $z = \cM(x)$ for some $x \in \cF$ and $\langle x, \cM^* y \rangle = \langle \cM(x), y \rangle = 0$. Therefore $ \cM(x) \in y^{\perp}$. Hence $z \in \cM(\cF) \cap y^{\perp}$ and $\cM(\cF \cap (\cM^*y)^{\perp}) \subseteq \cM(\cF) \cap \{ y\}^{\perp}$.
    
    For the other direction, let $z \in \cM(\cF) \cap \{ y\}^{\perp}$. Then $ z = \cM(x)$ for some $x \in \cF$ and $\langle \cM(x), y \rangle = \langle x, \cM^*y \rangle = 0$. Therefore $x \in (\cM^*y)^{\perp}$. Hence $ x \in \cF \cap (\cM^*y)^{\perp}$ and $z \in \cM(\cF \cap (\cM^*y)^{\perp})$. Therefore $\cM(\cF) \cap \{ y\}^{\perp} \subseteq \cM(\cF \cap (\cM^*y)^{\perp})$ and the proof is complete.
    \item $y_2 \in (\cM(\cF) \cap y_1^{\perp})^*$ is same as $y_2 \in (\cM(\cF \cap (\cM^*y_1)^{\perp}))^*$ by the first equality which is equivalent to $\cM^*y_2 \in (\cF \cap (\cM^*y_1)^{\perp})^*$.
    
    

    \item Assume $\cM^*y$ exposes a proper face $E$ of $\cK$, then $E = \cK \cap (\cM^*y)^{\perp}$ where $\cM^*y\in \cK^*$ and $\cM^*y \notin \cK^{\perp}$. So $y \in \cM(\cK)^*$ and there exists some $x \in \cK$ such that $\langle \cM^* y,  x \rangle = \langle y, \cM(x)\rangle \neq 0$. Hence $y \notin \cM(\cK)^{\perp}$. Therefore $F = \cM(E) = \cM(\cK) \cap y^{\perp}$ is a proper face of $\cM(\cK)$.
    
    For the other direction, assume $y$ exposes a proper face of $\cM(\cK)$, then $F = \cM(\cK) \cap y^{\perp}$ where $y \in \cM(\cK)^*$ and $y \notin \cM(\cK)^{\perp}$. Hence $\cM^*y \in \cK^*$ and there exists some $x \in \cK$ such that $\langle y, \cM(x) \rangle = \langle \cM^*y, x \rangle \notin 0$. So $\cM^*y \notin \cK^{\perp}$. Also $F = \cM(\cK) \cap y^{\perp} = \cM( \cK \cap (\cM^*y)^{\perp})$. $F$ is not empty implies $\cK \cap (\cM^*y)^{\perp}$ is not empty. So $\cM^*y$ exposes a proper face of $\cK$. 
\end{enumerate}
\end{proof}
\begin{theorem}\label{thm:singledeg}
Let $\cM: \E \to \bF$ be a linear transformation between two finite 
dimensional Euclidean spaces $\bE$ and $\bF$. 
Let $\cK \subset \E$ be a convex cone, $b\in \cF$ and
\[
\cF: = \{ X \in \cK: \cM(X) = b\}.
\]
Let $ v_1,  v_2, \cdots v_r \in b^{\perp}$, and $v_1 \in \cM(\cK)^*, v_2 \in (\cM(\cK)\cap v_1^{\perp})^*$, $\cdots$,$v_r \in (\cM(\cK)\cap \cdots \cap v_{r-1}^{\perp})^*$ and
\[
N := \cM(\cK)\cap v_1^{\perp}  \cap \cdots \cap
v_r^{\perp}, \quad
 E := \cK \cap \cM^*(v_1)^{\perp}  \cap \cdots \cM^*(v_r)^{\perp}.
\]
Then:
\begin{enumerate}

\item  $v_i$ exposes a face of $\cM(\cK) \cap v_1^{\perp}\cap \cdots \cap
v_{i-1}^{\perp}$ containing $b$ and  $\cM^*(v_i)$ exposes a  face of $\cK \cap \cM(v_1)^{\perp} \cap \cdots \cap \cM(v_{i-1})^{\perp}$ containing $F$.
\item  $N$ is a face of $\cM(\cK)$ containing $b$ and $E$ is a face of $\cK$ containing $F$;
\item $N = \face(b, \cM(\cK))$ if, and only if, $E = \face(F, \cK)$.
\end{enumerate}
\end{theorem}
\begin{proof}
\begin{enumerate}

\item 
Since $v_1 \in \cM(\cK)^*, v_2 \in (\cM(\cK)\cap v_1^{\perp})^*$, $\cdots$,$v_r \in (\cM(\cK)\cap \cdots \cap v_{r-1}^{\perp})^*$,   then by Lemma~\ref{lem:expose} we have $\cM^*v_1 \in \cK^*$, $\cM^*v_2 \in (\cK \cap (\cM^*y_1)^{\perp})^*$,$\cdots$,$\cM^*v_r \in (\cK\cap \cdots \cap (\cM^*v_{r-1})^{\perp})^*$. 
 Also $\langle v_i, b \rangle = \langle v_i, \cM(X) \rangle  = \langle \cM^* v_i, X \rangle = 0$ for any $X \in F$. 
So $v_1$ exposes a face of $\cM(\cK)$ containing $b$, $\cM^*v_1$ exposes a face of $\cK$ containg $F$, $v_2$ exposes a face of $\cM(\cK) \cap v_1^{\perp}$ containing $b$, $\cM^*v_2$ exposes a face of $\cK \cap (\cM^*y_1)^{\perp}$ containing $F$,... The conclusion follows through.


\item

 The conclusion follows from the property that a face of a face is also a face of the original cone.



\item Suppose $N = \face(b, \cM(\cK))$ and to the contrary $E $ is not the minimal face of $\cK$ containing $F$. Then by the theorem of alternative, there exists some nonzero $v$ such that $\cM^*v \in E^*$ exposing a proper face of $ E = \cK \cap \cM^*(v_1)^{\perp} \cap \cdots \cM^*(v_r)^{\perp}$ containing $F$. By Lemma~\ref{lem:expose}, $v$ exposes a proper face of $N = \cM(\cK) \cap v_1^{\perp} \cap \cdots \cap v_r^{\perp}$ containing $b$. However, this is impossible since $N$ is already the minimal face of $\cM(\cK)$ containing $b$. For the other direction, assume $E$ is the minimal face of $\cK$ containing $F$, suppose to the contrary that $N$ is not the minimal face of $\cM(\cK)$ containing $b$, then by Lemma~\ref{lem:exp}, there exists $v \in N^*$ exposing a proper face of $E$ containing $b$. Therefore $\cM^*v$ exposes a proper face of $E$ containing $F$, a contradiction.     

\end{enumerate}
\end{proof}
\begin{corollary}\label{cor:sing1}
The singularity degree of the system $\cF:= \{ X \in \cK: \cM(X) = b \}$ in $\bE$ is equal to the singularity degree of $N = \face(b, \cM(\cK))$ in $\bF$.
 \end{corollary}
\begin{proof}
This is a direct consequence from Lemma~\ref{lem:expose} and Theorem~\ref{thm:singledeg}.
\end{proof}

\subsection{The nullspace form}

We consider $\cF: = \{ X \in \cK: \cM(X) = b\}$ in its nullspace form,
i.e., let $\cF:= \{ Z \in \cK: \cA^*y + Z -C = 0 , y \in \cF)$ where $\cF$ is a finite dimensional linear space and $\cA$ is a linear transformation from $E$ to $\cF$. We define $L =  \cA^*(\cF)$ and let $\cK/L$ be the quotient cone and $C/L$ be a point on the quotient cone $\cK/L$. Then we have the following corollary.


\begin{corollary}\label{cor:sing2}
The singularity degree of the system $\cF:= \{ Z \in \cK: \cA^*y + Z -C = 0 , y \in \cF)$ is equal to the singularity degree of $\face(C/L, \cK/L)$, which is equal to the singularity degree of the nullspace form of the system $\mathcal{F}$.
 \end{corollary}
\begin{proof}
The first part of this theorem is just a reformulation of Theorem~\ref{thm:singledeg} in the nullspace form,  and the proof is similar to the proof of Theorem~\ref{thm:singledeg}.

For the second part, we just need to show there is an isomorphism between the two convex cones $\mathcal{M}(\cK)$ and $\cK / L$. For any $X$ such that $\mathcal{M}(X) = b$, $X$ can be written as $X = \cA^*y + C$ where $\mathcal{M}(\cA^* y) =0$ and $\mathcal{M}(C) = b$. The isomorphism is then given by $\mathcal{M}: \cK / L \rightarrow \mathcal{M}(\cK)$

Let $\bar{p}, \bar{q}$ be two distinct points on $\cK / L$ such that $\bar{p} = p + \cA^*y_1$ and $\bar{q} = q + \cA^*y_2$ such that $\mathcal{M}(p -q) \neq 0$. Then $\mathcal{M}(\bar{p}-\bar{q}) = \mathcal{M}(p-q) \neq 0$. So the homomorphism $\mathcal{M}$ from $\cK / L$ to $\mathcal{\cK}$ is one-to-one. It is easy to see it is also an onto map. Therefore it is an isomorphism.
\end{proof}



\begin{remark}
By combining Lemma~\ref{lem:min}, Proposition~\ref{prop:min} and Corollary~\ref{cor:sing1}, we can show that the minimum number of facial reduction in Algorithm~\ref{algo:frnonexp} is achieved by taking $d_i = \cM^*v_i \in \ri(\cF_i^* \cap \{ b\}^{\perp})$ at each step, a result which is proven in its nullspace form in ~\cite{lourencco2018facial}.
\end{remark}


\section{Bounding singularity degree for generic frameworks}
We begin with the definition of frameworks.
\begin{definition}\label{def:framework}
    A configuration $p: V \rightarrow \mathbb{R}^d$ of a graph $G = (V, E)$ assigns positions in $d$-dimensional Euclidean space to the vertices of $G$. A framework is an ordered pair $(G,p)$ consisting of a graph $G$ and a configuration $p$.
\end{definition}
Given a framework, one can define its stress and stress matrices. 
\begin{definition}
    A stress is an assignment $\omega: E \rightarrow \mathbb{R}$ of weights to the edges of $(G,p)$. An equilibrium stress is a stress such that 
    \[
    \sum_{j:v_iv_j \in E} \omega_{ij}(p(v_i) - p(v_j)) = 0.
    \]
    for all $v_i \in V$; i.e., the net stress at each vertex is 0.
\end{definition}
\begin{definition}
Given a stress $\omega$ for $(G, p)$ in $\mathbb{R}^d$, we define $\Omega = \Omega(\omega)$ to be the $|V| \times |V|$ symmetric matrix with
off-diagonal entries $-\omega_{ij}$ and diagonal entries $\sum_{j} \omega_{ij}$.
\end{definition}

In~\cite{connelly2015iterative}, the connection between facial reduction and level of PSD stress matrices is established. The sequence of stress matrices computed at each step of facial reduction algorithm provide a certificate of dimensional rigidity and universal rigidity. The minimum number of the level of stress matrices is indeed the singularity degree of the corresponding SDP optimization problem. 
See also the book~\cite{MR3887551}. We call the \emph{singularity degree of a framework} to be the minimal number of level stress matrices it has

So far, little has been known about the singularity degree of frameworks underlying a particular graph $G$. In~\cite{connelly2015iterative} the stress matrices and its connection to facial reduction is studied in the nullspace form. In~\cite{DrPaWo:14}, it is shown that the singularity degree of the matrix completion problem for chordal graphs is at most 1. We show that stress matrices are exactly the solutions to the auxiliary problem of the matrix completion problem in its nullspace form. Therefore, a framework underlying a chordal graph can have at most 1 level of stress matrices.

Given a framework $(G,p)$ where $G = (V,E)$, $|V|=n$,
let $d_{ij} = \| p_i - p_j \|^2, p_i\in \R^r, \, 
\forall i,j \in E$.  We let $d_E = \Pi_E (d)\in \R^E$. 
Let $\hat X$ be a (centered) feasible solution
for the following EDM completion problem:
\begin{equation}\label{prob:edmcomple}
	\mbox{find\; } \hat X \in \Ss_{+}^n \mbox{\;such that }
	\Pi_E(K(\hat X)) = d_E,\, \hat Xe = 0.
\end{equation}
Here $K(\hat X)$ is an EDM, a solution of the EDM completion problem, where $K(\hat X)_{ij} = X_{ii}+X_{jj} - 2 X_{ij}$.
If in addition $\hat X$ is a  \underline{unique} feasible solution,
then the framework $(G,p)$ is called \textdef{universally rigid}.
Necessarily, $\rank(\hat X)\leq r$.
We note that $\hat Xe=0$ and $\hat X=PP^T=PQQ^TP^T$, 
for any $Q$ orthogonal. Thus
the points are centered, $P^Te=0$, but we can allow rotations using $Q$.

The auxiliary problem from theorem of the alternative, 
that identifies positive singularity degree for \cref{prob:edmcomple}, 
is the following:
\[
0\neq Z = K^*(\Pi_E^*(v) )\succeq 0, \; \langle v, d \rangle = 0.
\]
The singularity degree of problem~\cref{prob:edmcomple} is called the \emph{singularity degree} of EDM completion.

In fact we can reformulate the 
EDM completion problem \cref{prob:edmcomple}
in the nullspace formulation.
We use the feasible $\hat X \in \Ss_+^n$ given above.
Now the problem is equivalent to finding $y$ such that 
\begin{equation}\label{prob:stress}
\hat X + \cN (y) = 
J \hat X J+ \sum_{ij\in E^c}y_{ij} J E_{ij}J \succeq 0,
\end{equation}
where we note that $\range(\cN) = \Null(\Pi_E\circ K)\cap \Ss_c^n$.

The auxiliary problem (theorem of the alternative) for
\cref{prob:stress} is 

\[
\begin{array}{rcll}
\langle \bar{\Omega},J\hat{X}J\rangle &=& 0 \\
\langle J E_{ij} J, \bar{\Omega} \rangle &=& 0, \forall ij \in E^c\\
\bar{\Omega} &\in&  \Ss_+^n  & \qquad \qquad 
\end{array}
\]
The centered solution $\Omega = \bar{J \Omega J}$ can be shown to be a PSD
equilibrium stress matrix. In the following, we will just refer PSD
equilibrium stress matrix as \emph{PSD stress matrix}. Since both $\hat{X}$ and $\Omega$ is centered, we have
\begin{equation}\label{eqn:rankXO}
    \rank(\hat{X}) + \rank (\Omega) \leq n -1.
\end{equation}

By doing facial reductions on problem~\ref{prob:stress}, we can obtain  a sequence of PSD stress matrices $\Omega_1, \Omega_2,...,\Omega_k$ by using the previous PSD stress matrix to reduce the problem to an equivalent problem over a smaller SDP cone, and the following relation holds. 
\begin{equation}\label{eqn:rankseq}
    \rank(\hat{X}) + \rank (\Omega_1)+...,+ \rank(\Omega_k) \leq n -1.
\end{equation}
\begin{lemma}\label{lem:edmcomple}
Given a framework $(G,p)$, The minimal number of (nonzero) PSD stress matrices of $(G,p)$ is equal to the singularity degree of the corresponding EDM completion problem.
\end{lemma}
\begin{proof}
The singularity degree of framework $(G,p)$ is the singularity degree of problem~\ref{prob:stress}, which is just the nullspace form of the EDM completion problem~\ref{prob:edmcomple}. Therefore the two singularity degrees are equal. 
\end{proof}

We can also characterize universal rigidity using sequence of PSD stress matrices. Given a framework $(G, p)$ whose affine span has dimension $d$. Then $(G, p)$ is called \emph{dimensionally rigid} if there exists a sequence of PSD stress matrices $\Omega_1, \Omega_2,...,\Omega_k$ such that 
\begin{equation}\label{eqn:rankseq}
    d + \rank (\Omega_1)+...,+ \rank(\Omega_k) = n -1.
\end{equation}
If in addition the member directions $\{ p_i - p_j \}_{\{i,j\} \in E_G}$ do not lie in a conic (there exists no non-trivial affine flex for $(G,p)$), then it is universally rigid~\cite{connelly2015iterative}.

 The singularity degree of the matrix completion problem is closely related to the stability of frameworks. A well-known bound for the forward error (distance to the orgininal solution set) in terms of the backward error (amount of constraints violation) of LMI is given in \cite{sturm2000error} in the form
\begin{equation}\label{eqn:errorbound}
    \mbox{forward error} = O((\mbox{backward error})^{2^{-s}})
\end{equation}
where $s$ is the singularity degree of the LMI.

Given a framework $(G,p)$ whose affine span has dimension $d$, if there exists at most one level of stress matrix $\Omega$ such that $\rank(\Omega) + d = n -1$, and the member directions $\{ p_i - p_j \}_{\{i,j\} \in E_G}$ do not lie in a conic (there exists no non-trivial affine flex for $(G,p)$), the framework is called \emph{super stable}\cite{connelly2002tensegrity,connelly2015iterative,connelly_guest_2022}. 
Hence, super stable framework is just universally rigid frameworks  with at most one level of PSD stress matrix.  It is conjectured in~\cite{alfakih2010universal} and proved in~\cite{gortler2014generic} that generic universal rigid frameworks has singularity at most 1. Hence, for generic frameworks, universally rigidity is equivalent to super stability. 
 




\subsection{Frameworks underlying chordal graphs}

Chordal graphs appear very often in the study of EDM completion problems. A graph is called chordal if any cycle of four or more nodes (vertices) has a chord,
i.e., an edge exists joining any two nodes that are not adjacent in the cycle.

The following fundamental theorem was proved by Bakonyi and Johnson~\cite{bakonyi1995euclidian}.

\begin{theorem}\label{thm:edmchordal}
Let $G$ be a chordal graph, A matrix $A \in \Ss_G^n$ has an EDM completion if and only if for every clique $\chi$ of the graph, the matrix $A_{\chi}$ is a Euclidean distance matrix.
\end{theorem}

The following theorem about the singularity of EDM completions of chordal graphs was proved in~\cite{DrPaWo:14}. 
\begin{theorem}\label{edm:chordal}
For chordal graphs, the singularity degree of the EDM completion problem is at most 1. 
\end{theorem}
We now have the following conclusion about the level of stress matrices of a framework underlying chordal graphs.

\begin{corollary}
Given a framework $(G,p)$. If $G$ is chordal, then the level of PSD stress matrices is at most 1. 
\end{corollary}
 \begin{proof}
Consequence of Lemma~\ref{lem:edmcomple} and Theorem~\ref{thm:edmchordal}.
 \end{proof}

\subsection{Frameworks where the underlying graph has a Laman subgraph}

Laman graphs have played in the study of rigidity theory. In particular, by restricting ourselves within generic configurations, a framework is minimally rigid if and only if the underlying graph is Laman\cite{laman1970graphs,connelly_guest_2022}.  

We start with the formal definition of Laman graphs.
\begin{definition}[Laman graphs] A graph $G = (V, E)$ is Laman if $|E| = 2|V| -3$ and every subgraph of $G$  with $k$ vertices has at most $2k -3$ edges. 
\end{definition}
\begin{definition}
Let $p$ be a configuration of $n$ points in $\mathbb{R}^d$. We say that $p$ is generic if the
coordinates of $p$ do not satisfy any polynomial with rational coefficients.
\end{definition} 
\begin{theorem}\label{thm:lamangeneric}
Let $G = (V,E)$ be a Laman-sparse graph with $n$ vertices, let $p = \{p_1,...,p_n \}$ be a 2d configuration of $G$, let $d_{ij} = \|p_i - p_j\|^2,\, \forall ij \in E$ be the squared distance between $p_i$ and $p_j$ and $d_E$ be a point in $\mathbb{R}^E$ whose coordinates are squared distances in the edge set $E$. If $p$ is generic, then $d_E$ is also generic.
\end{theorem}
\begin{proof}
Suppose to the contrary that $d_E$ is not generic, then there exists a nonzero polynomial $f$ such that $f(d_1,...,d_{|E|}) = 0$. Hence $f(...,||p_i-p_j||^2,...) = 0$, which we can rewrite it as $g(p_1,...,p_n) = 0 = f(...,||p_i-p_j||^2,...)$. Since $p$ is generic, we know $g \equiv 0$. ($g$ is a zero polynomial)

Since $g \equiv 0$, we know $f(d_1,...,d_{|E|}) =0$ for any point $(d_1,...,d_{|E|})$ where $(d_1,...,d_{|E|})$ are squared distances of some 2D framework $(G,p)$. Let $D_E$ be the set of all those points $(d_1,...,d_{|E|})$. Thus $f \equiv 0$ in the Zariski closure of $D_E$ (over the complex). Since $G$ is Laman-sparse, $E$ is independent in the algebraic matroid of the 2D Cayley-Menger ideal $\text{CM}(n,2)$\cite{malic2021combinatorial}.   Thus the Zariski closure of $D_E$ has dimension $|E|$. Since $f$ has $|E|$ variables, $f \equiv 0$ on a complex variety of  dimension $|E|$ implies $f$ is a zero polynomial, a contradiction. Hence $d_E$ has to be generic.
\end{proof} 


\begin{theorem}\label{thm:lamansingle}
Let $(G, p)$ be a generic framework of $n$ vertices in $\mathbb{R}^2$. If $G$ is Laman-sparse, then the level of stress matrices for $(G,p)$ is zero. In other words, if we are given the set of squared distance $d$ of $(G,p)$ corresponding to the edge set $E$ where $p$ is generic, then we can find a framework $\bar{p}$ of maximal dimension $n-1$ such that $||\bar{p}_i - \bar{p}_j||^2 = d_{ij},\,\forall ij \in E$. 
\end{theorem}
\begin{proof}
Since $(G,p)$ is a generic framework in $\mathbb{R}^2$, let $X_p$ be the corresponding gram matrix of $p$, by Theorem~\ref{thm:lamangeneric}, we know $d_E = \Pi_E(K(X_p))$ is also generic. Since $d_E$ is generic, we know the singularity degree of $\face(d_E, \Pi_E(K(\Ss_+^n)))$ is zero. Otherwise, suppose the singularity degree of $\face(d_E, \Pi_E(K(\Ss_+^n)))$ is 1 or more, then $d_E$ must lie on the boundary of $\Pi_E(K(\Ss_+^n)))$ . However, since $\Ss_+^n$ is a semialgebraic set, $\Pi_E(K(\Ss_+^n)))$ is also semialgebraic,  any point on the boundary of $\Pi_E(K(\Ss_+^n)))$  will be in the zero set of some polynomial $p(X)$ over $Q[X]$, therefore it must be non-generic, a contradiction. Therefore the singularity degree of $\face(d_E, \Pi_E(K(\Ss_+^n)))$ is zero. Thus by Corollary~\ref{cor:sing1} and Corollary~\ref{cor:sing2}, we know $(G,p)$ does not have a nonzero PSD stress matrix, hence the corresponding EDM completion problem admits a maximal rank solution $\bar{p}$.
\end{proof}

We now study the singularity degree of generic framework $(G,p)$ where $G$ is a Laman graph plus one edge. To begin with, we introduce the following lemma.
\begin{lemma}\label{lem:polyd}
 Let $f$ be a polynomial in $Q[x_1,...,x_n]$. If there exist $x, y$ with $x \neq y$ such that 
 \[
f(x) = f(y) = 0 \implies f(\lambda x +(1 - \lambda)y) = 0, \quad \forall \lambda \in [0,1]
\]
Then 
\[
f(\lambda x +(1 - \lambda)y) = 0, \quad \forall \lambda \in \mathbb{R}
\]
In other words, if the zero set of a polynomial contains a face, then it must contain the whole affine span of the face.
\end{lemma}
\begin{proof}
 Without loss, let $z = t x + (1 -t) y$ where $t >1$. Let $d = x -y$, and $z_0 = \frac{1}{2} x + \frac{1}{2}y$. So $f(z_0) = 0$ Then we can compute $f(z)$ by the Taylor expansion  at $z_0$ such that $f(z) = f(z_0) + \nabla_d f(z_0) (z - z_0) + \nabla^2_d f(z_0) (z - z_0)^2/2 + ...+ \nabla^{(m)}_d f(\bar{z})(z-z_0)^{m}/m! $.  It is easy to see $\nabla^{(i)}_df(z_0) = 0$ for any $i$. Also since $f$ is a polynomial, we have $\nabla^{(m)}_d f(\bar{z}) = 0$ for $m$ sufficiently large. Hence $f(z) = 0$, the lemma is proved. 
\end{proof}
We now introduce some basic notions from rigidity theory, which will be used later to derive an upper bound for the singularity degree of generic frameworks in $\R^2$ where the underlying graph has a Laman subgraph. 
\begin{definition}\label{def:rigidmtx}
The rigidity matrix of a framework $(G,p)$ in $\R^d$, denoted by $R(G,p)$, is an $|E| \times d|V|$ natrux where each row corresponds to an edge $(v_i,v_j) \in E$ and has the form 
\[
\left[ 
\begin{array}{c}
     \vdots  \\
     0 \cdots 0 \quad p(u) - p(v) \quad 0 \cdots 0 \quad p(v) - p(u) \quad 0 \cdots 0   \\
     \vdots
\end{array}
\right]
\]
where the given row corresponding to the edge $uv$, the d-tuple $p(u) -p(v)$ occurs in the columns for $u$ and the d-tuple $p(v) - p(u)$ occurs in the columns for $v$.
\end{definition}
\begin{definition}
The rigidity map of a framework $(G,p)$ is the map $f_G: \R^{d|V|} \rightarrow \R^{|E|}$ defined by 
\[
f_G(p) = (\cdots, ||p(v_i) - p(v_j)||^2, \cdots)_{(v_i,v_j) \in E},
\]
i.e., the map that returns the squared edge-lengths of a framework.
\end{definition}
Note that the rigidity matrix of a framework $(G,p)$ is the Jacobian matrix of the rigidity map divided by 2.
\begin{definition}
A framework in $\R^d$ is: infinitesimally rigid if $\rank(R(G,p)) =dn - \left( \begin{array}{c}
     d+1 \\
     2
\end{array} \right)$, infinitesimally flexible if $\rank R(G,p) < dn - \left( \begin{array}{c}
     d+1 \\
     2
\end{array} \right)$, independent if the rows of the rigidity matrix  are linearly independent, and minimally infinitesimally rigid if it is independent and infinitesimally rigid.
\end{definition}
Recall the following well-known theorem characterizing generically minimally rigid frameworks in $\R^2$.
\begin{theorem}[Pollaczek-Geiringer~\cite{pollaczek1927}]\label{thm:PGlaman}
A graph $G$ is generically minimally rigid in $\R^2$ if and only if $G$ is a Laman graph.
\end{theorem}
\begin{proposition}\label{prop:laman1}
A generic 2D framework $(G,p)$ where the underlying graph $G$ is a Laman graph plus one edge have singularity degree exactly equal to 1. 
\end{proposition}
\begin{proof}
Let $d_{E_1} = (d_1,\cdots, d_{|E_1|})$ be the point whose components are squared distances of a generic framework $(G, p)$ where the edge set of $G$ is $E_1 = E \cup e$. Assume $G$ is a Laman plus one graph and $E$ is the  set of the Laman graph while $e$ is the  extra one edge. 
We consider the rigidity map $f_G(\cdot)$ and Let $\bar{d}_{E_1} = f_G(\bar{p})$.

It is easy to see that the Jacobian matrix of the map $f_G$ is just the rigidity matrix in Definition~\ref{def:rigidmtx} multiplied by 2. 
Since $(G,\bar{p})$ is a generic framework, the Jacobian matrix of the map $f_G$ has maximum rank at $\bar{p}$ which is $2n -3$ by Theorem~\ref{thm:PGlaman}. By the constant rank theorem,  locally there is an open set $U$ in $\mathbb{R}^{2n}$ containing $\bar{p}$ such that $f_G(U)$ is a submanifold in $\mathbb{R}^{|E|}$ containing $\bar{d}_{E_1}$ of dimension $|E| = 2n - 3$. 
Hence $\bar{d}_{E_1}$ is on the boundary of  $\Pi_{E_1}(K(\Ss_+^n))$, so the singularity degree is at least 1. 

 By Tarski-Seidenberg theorem, $\Pi_{E_1}(K(\Ss_+^n))$ is a semialgebraic set generated by equalities and inequalities in $Q[X]$. Now suppose the minimal face containing $\bar{d}_{E_1}$ in $\Pi_{E_1}(K(\Ss_+^n))$ has singularity degree 2 or more.  Without loss assume $\bar{d}_{E_1} \in \mathcal{F}_2 \subsetneq \mathcal{F}_1 \subsetneq \Pi_{E_1}(K(\Ss_+^n))$ where $\mathcal{F}_1$ is a proper face of $\mathcal{F}_2$ and $\mathcal{F}_2$ is a proper face of $\Pi_{E_1}(K(\Ss_+^n))$. 
Then obviously $\mathcal{F}_1$ lies on the boundary of $\Pi_{E_1}(K(\Ss_+^n))$, so there exists some polynomial $f_1(X)$ in $Q[X]$ such that $f_1(\mathcal{F}_1) = 0$.  Pick a point $d^*$ in the relative interior of $\mathcal{F}_1$, then the line $L$ connecting $d^*$ and $\bar{d}_{E_1}$ is contained in the zero set of $f_1$ by Lemma~\ref{lem:polyd}. Now since $\mathcal{F}_2$ is on the boundary of $\mathcal{F}_1$,  $\bar{d}_{E_1}$ is also on the boundary of $\mathcal{F}_1$, we take all the polynomials defining the boundary of the semialgebraic set $\Pi_{E_1}(K(\Ss_+^n))$ that vanish on $\mathcal{F}_2$, among them there must exist another polynomial $f_2(X)$ in $Q[X]$ such that the zero set of $f_2$ intersects $L$ properly at $\bar{d}_{E_1}$ (otherwise the zero set of $f_2$ contains the entire line $L$, therefore $\bar{d}_{E_1}$ can not be on the boundary of $\mathcal{F}_1$). Since the zero set of $f_1$ contains the line $L$ while the zero set of $f_2$ intersects $L$ properly,  the local dimension of the intersection of the zero sets of $f_1$ and $f_2$ at $\bar{d}_{E_1}$ is at less one less than the local dimension of the zero sets of $f_1$ at $\bar{d}_{E_1}$. Hence the dimension of the intersection of the zero sets of $f_1$ and $f_2$ at $\bar{d}_{E_1}$ is at most $2n-4$ locally. 

Now since $f_1(\bar{d}_{E_1}) = 0$, we have $f_1(...,||\bar{p}_i - \bar{p}_j||^2,...) = g_1(\bar{p}_1,...,\bar{p}_n)= 0 $. 
Since $\bar{p}$ is algebraic independent over $Q$, $g_1$ must be a zero polynomial. Therefore 
  pick any $d^* \in f_G(U)$, we have $f_1(d^*) = f_1(...,||p^*_i - p^*_j||^2,...) = g_1(p^*_1,...,p^*_n) = 0$. Similarly $f_2(d^*) = f_2(...,||p^*_i - p^*_j||^2,...) = g_2(p^*_1,...,p^*_n) = 0$. However $f_G(U)$ has dimension $2n -3$, a contradiction to the previous conclusion that the local dimension of the intersection of $f_1$ and $f_2$ at $\bar{d}_{E_1}$ is at most $2n-4$ locally.   
\end{proof}
\begin{remark}
In~\cite{bernstein2021maximum}, it is shown that if $G$ is a Laman circuit, then there are generic 2-dimensional frameworks $(G,p)$ satisfying a non-zero PSD stress. So we have shown that in this case, there is exactly one level of non-zero PSD stress for a generic 2-dimensional framework $(G,p)$.
\end{remark}
\begin{proposition}\label{prop:lamand}
A generic 2D framework $(G,p)$ where the underlying graph $G$ is a Laman graph plus $d$ edges have singularity at most $d$.
\end{proposition}

\begin{proof}
The proof is similar to the proof of the previous theorem. 
Let $\bar{d}_{E_1} = (d_1,\cdots, d_{|E_1|})$ be the point whose components are squared distances of a generic framework $(G, \bar{p})$ where the edge set of $G$ is $E_1 = E \cup E_d$. Assume $G$ is a Laman plus $d$ graph and $E$ is the set set of the Laman graph while $E_d$ is the  extra $d$ edges. 
Again let $f_G: p \rightarrow d_{E_1}$ to be the rigidity map from frameworks to squared distances and let $\bar{d}_{E_1} = f_G(\bar{p})$.

          Now suppose the minimal face containing $\bar{d}_{E_1}$ in $\Pi_{E_1}(K(\Ss_+^n))$ has singularity degree $d+1$ or more.  Without loss assume $\bar{d}_{E_1} \in \mathcal{F}_{d+1} \subseteq \cdots \subsetneq \mathcal{F}_1 \subsetneq \Pi_{E_1}(K(\Ss_+^n))$. By a similar argument, we can find polynomials $f_1,...,f_{d+1}$ in $Q[X]$ and we claim the intersection of $f_1,...,f_{d+1}$ at $\bar{d}_{E_1}$ has dimension at most $2n -4$. We prove the claim by induction. We assume $\mathcal{F}_i$ is contained in the zero set of $f_1,..,f_i$ and the intersection of $f_1,..,f_i$ at $\bar{d}_{E_1}$ has dimension at most $2n-3 + d - i$ locally. Now $\mathcal{F}_{i+1}$ is on the boundary of $\mathcal{F}_i$. Pick a point $d^*$ in the relative interior of $\mathcal{F}_i$ and let $L$ be the line connecting $d^*$ and $\bar{d}_{E_1}$. Then there exists a polynomial $f_{i+1}$ which vanishes on $\mathcal{F}_{i+1}$ such that $f_{i+1}$ intersects $L$ properly. Hence the dimension of the intersection of $f_1,...,f_i$ with $f_{i+1}$ is at least one less than the intersection of $f_1,..,f_i$ at $\bar{d}_{E_1}$. Therefore the intersection of $f_1,...,f_i, f_{i+1}$ at $\bar{d}_{E_1}$ has dimension at most $2n-3+ d - i- 1$. The induction hypothesis holds. Hence the intersection of $f_1,..., f_{d+1}$ at $\bar{d}_{E_1}$ has dimension at most $2n-3+ d - d-1 = 2n-4$. 
          
         Again  since $(G,\bar{p})$ is a generic framework, the Jacobian matrix of the map $f_G$ has maximum rank at $\bar{p}$ which is $2n -3$.  So locally there is an open set $U$ containing $\bar{p}$ such that $f_G(U)$ is a submanifold containing $\bar{d}_{E_1}$ of dimension $|E| = 2n - 3$. Then $f_1(\bar{d}_{E_1})  = f_2(\bar{d}_{E_1}) = ... = f_{d+1}(\bar{d}_{E_1}) = 0$ and due to algebraic independence of $\bar{p}$ we have $f_1(d^*)  = f_2(d^*) = ... = f_{d+1}(d^*) = 0$ for any $\bar{d}^* \in f_G(U)$. A contradiction to the previous conclusion that the intersection of $f_1,..., f_{d+1}$ at $\bar{d}_{E_1}$ has dimension at most $2n-3+ d - d-1 = 2n-4$ locally.
\end{proof}

\emph{Example 1. } Consider the EDM completion problem where $G$ is $K_4$
\[
\begin{array}{rcl}
     \mbox{find \;} X & \succeq & 0 \\
    \mbox{s.t. } \Pi_E(K(X)) &=& d_E \\
\end{array}
\]
where $E = \{12,13,14,23,24,34\}$, $X \in \Ss_+^4$ and $d_E$ is the set of distances from a generic configuration $P = [p_1,p_2,p_3,p_4]^T$ as shown in Figure~\ref{fig:pic2}. 
\begin{figure}[h]
    \centering
    \includegraphics[width=0.5\textwidth,height=0.3\textheight]{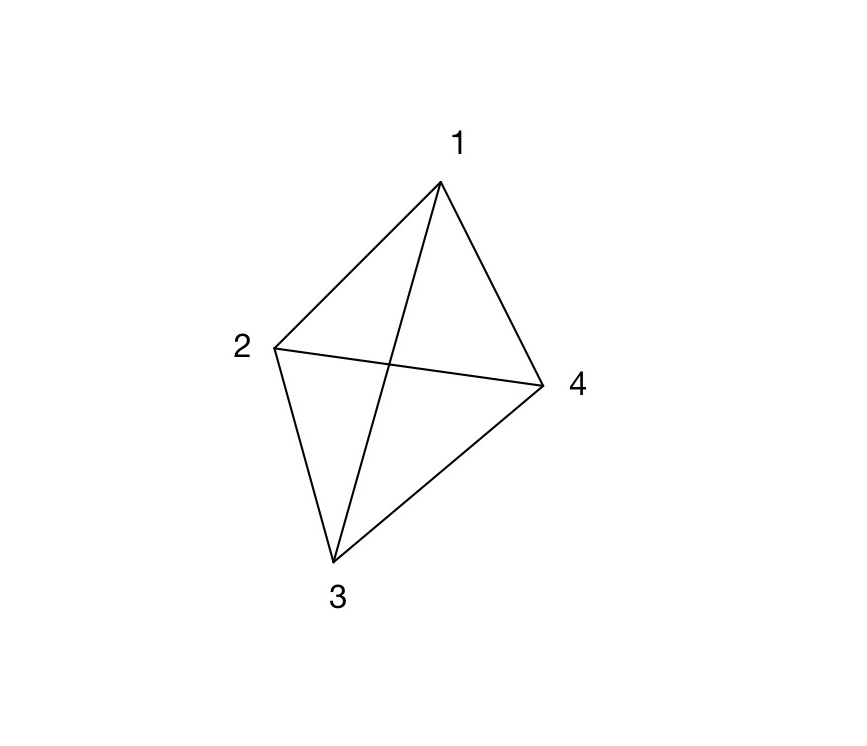}
    \caption{framework underlying $K_4$}
    \label{fig:pic2}
\end{figure}
The stress matrix has the following form
\[
\cF = \left( 
\begin{array}{cccc}
    w_{12} + w_{13} + w_{14} & -w_{12} & -w_{13} & -w_{14} \\
    -w_{12} & w_{12} + w_{23} + w_{24} & -w_{23} & -w_{24} \\ 
    -w_{13} & -w_{23} & w_{13} + w_{23}+w_{34} & -w_{34} \\
    -w_{14} & -w_{24} & -w_{34} & w_{14}+w_{24}+w_{34}
\end{array}
\right)
\]
where $w_{12},w_{23},w_{14},w_{34}$ are outward stresses (positive signs) and $w_{24},w_{13}$ are negative stresses (negative signs).

It can be seen there must exist at least one nonzero PSD stress matrix as the framework with distance $d_E$ can not be lifted to higher dimension. If vertex 1 is rotated around the edge $24$ to lift the framework to dimension 3, the distance between vertex 1 and vertex 3 would shrink. Also since $\rank(PP^T) = 2$ and the rank of a nonzero PSD stress matrix is at least 1, we have $\rank (PP^T) + \rank(\cF) = 3 = n -1$ where $n = 4$ is the number of vertices. By \eqref{eqn:rankXO},  $\rank (PP^T) + \rank(\cF)$ is already maximal, hence there is only 1 level of PSD stress matrix.

\subsection{Singularity of frameworks in $\R^3$}
The techniques we have developed so far to analyze generic frameworks in $\R^2$ can be easily adapted to deal with generic frameworks in $\R^3$.   Recall the following theorem characterizing generically minimally rigid graphs in $\R^3$.
\begin{theorem}[Gluck~\cite{gluck1975almost}]
Every maximal planar graph is generically minimally rigid in $\R^3$.
\end{theorem}
Let $(G,p)$ be a generical framework in $\R^3$. Then $|E| = 3|V| - 6$ and the rigidity matrix $R(G,p)$ has full rank with $3|V| - 6$ linearly independent rows. Similarly, one can prove the following theorem regarding singularity degree of generic frameworks in $\R^3$. 
\begin{proposition}\label{prop:laman3d}
A generic 3D framework $(G,p)$ where the underlying graph $G$ is a maximal planar graph  plus $d$ edges have singularity at most $d$.
\end{proposition}
\begin{proof}
The proof follows similarly from the proof of Proposition~\ref{prop:lamand}.
\end{proof}
\subsection{Connection to dimensional and universally rigidity}

We recall the following theorem characterizing generic universally rigid framework
\begin{theorem}~\cite{gortler2014generic}\label{thm:univrigid}
Suppose $G$ is a graph with $v$ vertices or more vertices and $p$ is a generic configuration in $\R^d$, suppose $v \geq d+2$. Suppose there is a PSD equilibrium stress matrix $\Omega$ such that $\rank(\Omega)= v - d -1$. Then $(G,p)$ is universally rigid.
\end{theorem}

We are now ready to prove that for generic frameworks under certain graphical structure, dimensional rigidity and universal rigidity are equivalent.
\begin{theorem}
Given a generic framework $(G,p)$ in $\R^2$ where $G$ is a Laman graph plus one edge, or a generic framework in $\R^3$ where $G$ is a maximal planar graph plus one edge,  if $(G,p)$ is dimensionally rigid, then $(G,p)$ must be universally rigid.
\end{theorem}
\begin{proof}
    Since $p$ is dimensional rigid, and $(G,p)$ has singularity degree at most one by Proposition~\ref{prop:lamand} and Proposition~\ref{prop:laman3d}, we know that the PSD equilibrium stress matrix has maximal rank $v-d-1$. Therefore $(G,p)$ is universally rigid by Theorem~\ref{thm:univrigid}.
\end{proof}
\section{Bounding singularity degree for tensegrities}
The above results can be easily extended to linear conic optimization problems where inequalities are included by adding auxiliary variables.  
\subsection{Tensegrity}
A tensegrity $(G,p)$ is a framework where some edges (bars) are replaced by cables, which are constrained to not get longer in length, and some 
bars are replaced by struts, which are constrained to not get shorter in length. Given a graph $G$ with edge sets for bars, cables and struts, a tensegrity underlying $G$ is a feasible solution of the following linear SDP feasibility problem.
\begin{equation}\label{prob:edmtensegrity}
\begin{array}{cc}
   	\mbox{find\; } \hat X \in \Ss_{+}^n  \\
 \mbox{\;such that }
	\Pi_{E_0}(K(\hat X)) = d_{E_0},\, \\
	\Pi_{E_1}(K(\hat X)) \leq d_{E_1},\,\\
	\Pi_{E_2}(K(\hat X)) \geq d_{E_2},\, \hat Xe = 0.
\end{array}
\end{equation}
where $E = E_0 \cup E_1 \cup E_2$,  $E_0$ is the edge set for bars, $E_1$ is the edge sets for cable , and $E_2$ is the edge set for struts.

The auxiliary problem for $\eqref{prob:edmtensegrity}$ is then 
\[
\begin{array}{rcll}
\langle \bar{\Omega},J\hat{X}J\rangle &=& 0 \\
\langle J E_{ij} J, \bar{\Omega} \rangle &=& 0, \forall ij \in (E_0 \cup E_1 \cup E_2)^c\\
\langle J E_{ij} J, \bar{\Omega} \rangle & \geq & 0, \forall ij \in E_1\\
\langle J E_{ij} J, \bar{\Omega} \rangle & \leq & 0, \forall ij \in E_2\\
\bar{\Omega} &\in&  \Ss_+^n  & \qquad \qquad 
\end{array}
\]
A PSD stress matrix $\cF = J \bar{\cF} J$ is called \emph{proper} if it satisfies the above auxiliary problem or equivalently $\cF$ is a PSD stress matrix for the framework and $\cF_{ij} \geq 0$, when $\{ i,j\}$ is a cable, and  $\cF_{ij} \leq 0$, when $\{ i,j\}$ is a strut.  

We add auxiliary variables to problem \eqref{prob:edmtensegrity} 
\begin{equation}\label{prob:edmtensegrityauxi}
\begin{array}{cc}
   	\mbox{find\; } \hat X \in \Ss_{+}^n  \\
 \mbox{\;such that }
	\Pi_{E_0}(K(\hat X)) = d_{E_0},\, \\
	\Pi_{E_1}(K(\hat X)) + z_{E_1} =d_{E_1},\,\\
	\Pi_{E_2}(K(\hat X)) - z_{E_2} =  d_{E_2},\, \hat Xe = 0. \\
	z_{E_1} \geq 0, z_{E_2} \geq 0
\end{array}
\end{equation}
Let $\cK_2 = (0^{|E_0|},\R^{|E_1|}_{\geq 0}, \R^{|E_2|}_{\leq 0})$ and $\cK_1 =  \Pi_{E}(K(\Ss_+^n))$.
The singularity degree of problem ~\eqref{prob:edmtensegrity} is then the minimal face containing $(d_{E_0}, d_{E_1}, d_{E_2})^T$ over the cone $\cK_1 + \cK_2$ where the $+$ sign is the Minkowski sum. It is easy to see $\cK_1 + \cK_2$ is a semialgebraic set since $\cK_1$ is semialgebraic. The boundary of  $\cK_1 + \cK_2$ is generated by polynomials in $Q[X]$. 

\begin{proposition}\label{thm:lamandtensegrity}
A generic tensegrity framework $(G,p)$ in $\R^2$ where the underlying graph $G$ is a Laman graph plus $d$ edges, or a generic tensegrity framework $(G,p)$ in $\R^3$ where the underlying graph $G$ is a maximal planar graph plus $d$ edges have singularity at most $d$.
\end{proposition}
\begin{proof}
The proof follows by replacing $\Pi_{E_1}(K(\Ss_+^n))$ with  $\cK_1 + \cK_2$ in the proof of Proposition~\ref{prop:lamand}.
\end{proof}

\section{Conclusion}

We define the singularity degree of any face of a convex cone as the minimum number of steps to expose it by using exposing vectors in a sequence of growing dual cones.  Given a convex cone, the linear image of it is also a convex cone.  However, the linear image of a facially exposed convex cone is not necessarily facially exposed.  We show that the singularity degree of the minimum face containing vector $b$ is exactly the singularity degree of the linear conic optimization problem $\{ \min c^T x \;|\; \cM(x) = b, x \in \cK\}$. We then utilize tools from algebraic geometry and rigidity theory to give an upper bound of the singularity degree of the EDM completion problem related to  genericframeworks or tensigrities in the Euclidean space $\R^d$ where $d = 2$ or $3$. The most interesting case seems to be the case where $G$ is a Laman graph plus 2 edge in $\R^2$, as this is the first case where strict complementarity could possibly fail for a generic framework in $\R^2$. So far, all the frameworks where strict complementarity fails in the literature are constructed in a way that the coordinates of some points are in some special positions, i.e., the framework is non-generic. It is an interesting open problem to construct generic frameworks with singularity degree more than 1.   Although our method is focused on EDM completion problem, it can be adjusted to be analyze other similar SDP problems. 

\bibliographystyle{siam}

\printindex
\addcontentsline{toc}{section}{Index}

\cleardoublepage
\bibliography{reference}

\begin{thebibliography}{10}

\bibitem{alfakih2010universal}
{\sc A.~Y. Alfakih}, {\em On the universal rigidity of generic bar frameworks},
  Contributions to Discrete Mathematics, 5 (2010).

\bibitem{MR3887551}
{\sc A.~Y. Alfakih}, {\em Euclidean distance matrices and their applications in
  rigidity theory}, Springer, Cham, 2018.

\bibitem{bakonyi1995euclidian}
{\sc M.~Bakonyi and C.~R. Johnson}, {\em The euclidian distance matrix
  completion problem}, SIAM Journal on Matrix Analysis and Applications, 16
  (1995), pp.~646--654.

\bibitem{bernstein2021maximum}
{\sc D.~I. Bernstein, S.~Dewar, S.~J. Gortler, A.~Nixon, M.~Sitharam, and
  L.~Theran}, {\em Maximum likelihood thresholds via graph rigidity}, arXiv
  preprint arXiv:2108.02185,  (2021).

\bibitem{borwein1981regularizing}
{\sc J.~Borwein and H.~Wolkowicz}, {\em Regularizing the abstract convex
  program}, Journal of Mathematical Analysis and Applications, 83 (1981),
  pp.~495--530.

\bibitem{connelly2002tensegrity}
{\sc R.~Connelly}, {\em Tensegrity structures: why are they stable?}, in
  Rigidity theory and applications, Springer, 2002, pp.~47--54.

\bibitem{connelly2015iterative}
{\sc R.~Connelly and S.~J. Gortler}, {\em Iterative universal rigidity},
  Discrete \& Computational Geometry, 53 (2015), pp.~847--877.

\bibitem{connelly_guest_2022}
{\sc R.~Connelly and S.~D. Guest}, {\em Frameworks, Tensegrities, and
  Symmetry}, Cambridge University Press, 2022.

\bibitem{DrPaWo:14}
{\sc D.~Drusvyatskiy, G.~Pataki, and H.~Wolkowicz}, {\em Coordinate shadows of
  semidefinite and euclidean distance matrices}, SIAM Journal on Optimization,
  25 (2015), pp.~1160--1178.

\bibitem{DrusWolk:16}
{\sc D.~Drusvyatskiy and H.~Wolkowicz}, {\em The many faces of degeneracy in
  conic optimization}, Foundations and
  Trends\textsuperscript{\tiny\textregistered} in Optimization, 3 (2017),
  pp.~77--170.

\bibitem{gluck1975almost}
{\sc H.~Gluck}, {\em Almost all simply connected closed surfaces are rigid}, in
  Geometric topology, Springer, 1975, pp.~225--239.

\bibitem{gortler2014generic}
{\sc S.~J. Gortler and D.~P. Thurston}, {\em Characterizing the universal
  rigidity of generic frameworks}, Discrete \& Computational Geometry, 51
  (2014), pp.~1017--1036.

\bibitem{hiriart2004}
{\sc J.-B. Hiriart-Urruty and C.~Lemar{\'e}chal}, {\em Fundamentals of convex
  analysis}, Springer Science \& Business Media, 2004.

\bibitem{laman1970graphs}
{\sc G.~Laman}, {\em On graphs and rigidity of plane skeletal structures},
  Journal of Engineering mathematics, 4 (1970), pp.~331--340.

\bibitem{lourencco2018facial}
{\sc B.~F. Louren{\c{c}}o, M.~Muramatsu, and T.~Tsuchiya}, {\em Facial
  reduction and partial polyhedrality}, SIAM Journal on Optimization, 28
  (2018), pp.~2304--2326.

\bibitem{malic2021combinatorial}
{\sc G.~Mali{\'c} and I.~Streinu}, {\em Combinatorial resultants in the
  algebraic rigidity matroid}, arXiv preprint arXiv:2103.08432,  (2021).

\bibitem{Pataki:07}
{\sc G.~Pataki}, {\em On the closedness of the linear image of a closed convex
  cone}, Math. Oper. Res., 32 (2007), pp.~395--412.

\bibitem{pataki2013strong}
{\sc G.~Pataki}, {\em Strong duality in conic linear programming: facial
  reduction and extended duals}, in Computational and analytical mathematics,
  Springer, 2013, pp.~613--634.

\bibitem{pollaczek1927}
{\sc H.~Pollaczek-Geiringer}, {\em {\"U}ber die gliederung ebener fachwerke},
  ZAMM-Journal of Applied Mathematics and Mechanics/Zeitschrift f{\"u}r
  Angewandte Mathematik und Mechanik, 7 (1927), pp.~58--72.

\bibitem{ramana1997exact}
{\sc M.~V. Ramana}, {\em An exact duality theory for semidefinite programming
  and its complexity implications}, Mathematical Programming, 77 (1997),
  pp.~129--162.

\bibitem{ramana1997strong}
{\sc M.~V. Ramana, L.~Tun{\c{c}}el, and H.~Wolkowicz}, {\em Strong duality for
  semidefinite programming}, SIAM Journal on Optimization, 7 (1997),
  pp.~641--662.

\bibitem{sremac2021error}
{\sc S.~Sremac, H.~J. Woerdeman, and H.~Wolkowicz}, {\em Error bounds and
  singularity degree in semidefinite programming}, SIAM Journal on
  Optimization, 31 (2021), pp.~812--836.

\bibitem{sturm2000error}
{\sc J.~F. Sturm}, {\em Error bounds for linear matrix inequalities}, SIAM
  Journal on Optimization, 10 (2000), pp.~1228--1248.

\bibitem{tanigawa2017singularity}
{\sc S.~Tanigawa}, {\em Singularity degree of the positive semidefinite matrix
  completion problem}, SIAM Journal on Optimization, 27 (2017), pp.~986--1009.

\end{thebibliography}
\addcontentsline{toc}{section}{Bibliography}

\end{document}